\documentclass{amsart}

\title[Degenerate Pullback Attractors for the 3D NSE]
      {Degenerate Pullback Attractors for the 3D Navier-Stokes Equations}
\author{Alexey Cheskidov}
    \address{Department of Mathematics, Statistics, and Computer Science\\
        University of Illinois at Chicago\\
        322 Science and Engineering Offices (M/C 249)\\
        851 S. Morgan Street\\
        Chicago, Illinois 60607-7045 USA}
    \email{acheskid@uic.edu}
\author{Landon Kavlie}
    \address{Department of Mathematics, Statistics, and Computer Science\\
        University of Illinois at Chicago\\
        322 Science and Engineering Offices (M/C 249)\\
        851 S. Morgan Street\\
        Chicago, Illinois 60607-7045 USA}
    \email{lkavli2@uic.edu}
\thanks{The authors were partially supported by NSF Grant DMS-1108864.}

\usepackage{amsthm}
\usepackage{amsfonts}
\usepackage{amsmath}
\usepackage{mathrsfs}
\usepackage{graphicx}
\usepackage[colorlinks=true]{hyperref}
\hypersetup{urlcolor=blue, citecolor=red}

\newcommand{\ddx}{\mathrm{dx}}
\newcommand{\ddxi}{\mathrm{d}\xi}
\newcommand{\derivt}{\mathrm{\frac{d}{dt}}}

\newcommand{\ds}{\mathrm{d_s}}
\newcommand{\dw}{\mathrm{d_w}}
\newcommand{\dd}{\mathrm{d_\bullet}}
\newcommand{\ee}{\epsilon}

\newcommand{\WWd}{\Omega_\bullet}
\newcommand{\WWs}{\Omega_\mathrm{s}}
\newcommand{\WWw}{\Omega_\mathrm{w}}

\newcommand{\N}{\mathbb{N}}
\newcommand{\Z}{\mathbb{Z}}
\newcommand{\R}{\mathbb{R}}

\newcommand{\A}{\mathscr{A}}

\newcommand{\AAw}{\mathscr{A}_\mathrm{w}}
\newcommand{\AAd}{\mathscr{A}_\bullet}

\newcommand{\Xw}{X_\mathrm{w}}
\newcommand{\Xd}{X_\bullet}

\newcommand{\BB}{\mathscr{B}}
\newcommand{\EE}{\mathscr{E}}
\newcommand{\FF}{\mathscr{F}}

\newcommand{\PP}{\mathscr{P}}
\newcommand{\TT}{\mathscr{T}}

\newcommand{\abs}[1]{\lvert#1\rvert}
\newcommand{\norm}[1]{\lVert#1\rVert}
\newcommand{\ip}[2]{\left(#1,#2\right)}
\newcommand{\IP}[2]{\left(\left(#1,#2\right)\right)}
\newcommand{\angip}[2]{\langle#1,#2\rangle}

\newcommand{\grad}{\nabla}

\newcommand{\numberthis}{\addtocounter{equation}{1}\tag{\theequation}}

\newcounter{icount}
\newcommand{\refitem}[1]{\item[#1]\refstepcounter{icount}\label{#1}}
\newcommand{\iref}[1]{\hyperref[#1]{#1}}

\begin{document}

\newtheorem{thm}{Theorem}[section]
\newtheorem{prop}[thm]{Proposition}
\newtheorem{lem}[thm]{Lemma}
\newtheorem{cor}[thm]{Corollary}
\newtheorem{define}[thm]{Definition}
\newtheorem{rmk}[thm]{Remark}

\begin{abstract}
As was found in \cite{CK13}, the 3D Navier-Stokes equations with a translationally bounded force possesses pullback attractors $\AAw(t)$ in a weak sense. Moreover, those attractors consist of complete bounded trajectories. In this paper, we present a sufficient condition under which the pullback attractors are degenerate. That is, if the Grashof number is small enough, each section of the pullback attractor is a single point on a unique, complete, bounded, strong solution. We then apply our results to provide a new proof of the existence of a unique, strong, periodic solution to the 3D Navier-Stokes with a small, periodic forcing term.
\end{abstract}

\maketitle

\section{Introduction}

A natural question in the study of attractors for dissipative partial differential equations is what conditions on the force necessitate a trivial attractor. That is, under what conditions on the force do we find that the attractor $\A=\{z\}$, a single point. This is closely related to the question of dimensionality of the attractor. For the Navier-Stokes equations, it has long been known that they possess a compact global attractor in two dimensions (\cite{FP67}). The dimension of this global attractor is controlled by the Grashof number $G=\frac{\norm{f}_2^2}{\nu^2\lambda_1}$ (\cite{CF85}, \cite{CFT85}). In particular, when the Grashof number is small enough, the attractor is trivial. For a proof of this fact, see the book \cite{CV02}, although the argument used goes back to \cite{L63}. That is, $\A=\{z\}$ where $z$ is the unique stationary solution to the Stokes system. An analogous result was proven by Chepyzhov and Vishik using trajectory attractors in three dimensions where the Grashof number is given by $G=\frac{\norm{f}_2^2}{\nu^2\lambda_1^{3/4}}$ (\cite{CV02}). This result can easily be extended to the theory of weak attractors as developed in \cite{FT87}, \cite{CF06}, \cite{C09}.

In the nonautonomous Navier-Stokes equations, we have that $f=f(t)$ depends on time. In this situation, we consider the pullback attractor for the system. That is, a family of minimal closed sets $\A(t)$ which uniformly attract all bounded subsets of the phase space in a pullback sense. We will rigorously define these concepts in Section~\ref{setup}, below. For more information on the existence and structure of the pullback attractor in two dimensions, we refer the reader to the book \cite{CLR13}. For the existence and structure of pullback attractors for the three dimensional case, we refer the reader to \cite{CK13}. Now, in the book by Carvalho, Langa, and Robinson \cite{CLR13}, they produce a theorem giving sufficient conditions under which the pullback attractor $\A(t)$ for the 2D Navier-Stokes equations is a single point. They find that if some form of the Grashof number is small enough, then the pullback attractor is degenerate. That is, if 
\begin{equation*}
G(t):=\frac{1}{\nu^2\lambda_1}\left(\limsup_{s\rightarrow-\infty}\frac{1}{t-s}
            \int_s^t\norm{f(\xi)}_2^2\ddxi\right)^{1/2}
\end{equation*}
is small enough, then the pullback attractor $\A(t)$ is trivial. We present an analogous result for the 3D Navier-Stokes equations assuming a translationally bounded force in $L^2_{loc}(\R,L^2)$. We show that if a form of the Grashof number is small enough, then the weak pullback attractor $\AAw(t)=\{v(t)\}$ for a complete, bounded solution $v$. 

The paper is laid out as follows: Section~\ref{background} is devoted to recalling the basic definitions and setup of our problem. We recall the definition of generalized evolutionary systems and pullback attractors as they first appeared in \cite{CK13}. We then recall the major theorems of existence and structure for the 3D Navier-Stokes equations. Section~\ref{trivial} is then devoted to our proof of the triviality of the pullback attractors under the assumption of small enough Grashof number. We start by proving the strongness of the trajectories on the pullback attractor in Section~\ref{strong}. Then, in Section~\ref{serrin}, we use a modification of Serrin's argument on the uniqueness of strong solutions in intervals of regularity (\cite{S62}) to prove that the pullback attractor must be a single point under the usual smallness assumption of the force. 

We give a particularly interesting application of our results in Section~\ref{periodic}. Here, we apply the theorem giving us a unique, bounded, strong solution to the case where the force is periodic. In this setting, we show that this solution is periodic. Theorems proving the existence of periodic solutions to the 3D Navier-Stokes equations go back to Serrin (\cite{S59}). Additional results of this type are given in \cite{J60}, \cite{P60}, \cite{GS04} among others. For a more exhaustive discussion of the history of these results, see the recent paper by Kyed (\cite{K13}). A common technique in the existence of periodic solutions is the use of Poincar\'{e} maps and fixed-point arguments. Instead, we use the structure of the pullback attractor to prove the existence of a periodic solution.

In all that follows, we use the usual convections of $c_0$, $c_1$, $\dots$ for particular (fixed) constants. On the other hand, the constant $C$ will change from line to line.

\section{Generalized Evolutionary Systems}
\label{background}
\subsection{Setup and Previous Results}
\label{setup}

We begin with the setup and definition of a generalized evolutionary system as it appeared in \cite{CK13}. So, let $(X,\ds(\cdot,\cdot))$ be a metric space with a metric $\ds$ known as the strong metric on $X$. Let $\dw$ be another metric on $X$ satisfying the following conditions:
\begin{enumerate}
\item $X$ is $\dw$ compact. 
\item If $\ds(u_n,v_n)\rightarrow0$ as $n\rightarrow\infty$ for some $u_n,v_n
\in X$ then $\dw(u_n,v_n)\rightarrow0$ as $n\rightarrow\infty$.
\end{enumerate}
As justified by property (2), $\dw$ is called the weak metric on $X$. For simplicity, denote by $\Xd$ the set $X$ with the topology induced by the metric $\dd$. Next, denote by $\overline{A}^\bullet$ the closure of the set $A\subseteq X$ in the topology generated by $\dd$. Note that any strongly compact set ($\ds$-compact) is also weakly compact ($\dw$-compact), and any weakly closed set ($\dw$-closed) is also strongly closed ($\ds$-closed).

Let $C([a,b];X_\bullet)$, where $\bullet=$ s or w, be the space of $\dd$-continuous $X$-valued functions on $[a,b]$ endowed with the metric
\begin{equation*}
\mathrm{d}_{C([a,b];X_\bullet)}(u,v):=\sup_{t\in[a,b]}\dd(u(t),v(t)).
\end{equation*}
Let also $C([a,\infty);X_\bullet)$ be the space of all $\dd$-continuous $X$-valued functions on $[a,\infty)$ endowed with the metric
\begin{equation*}
\mathrm{d}_{C([a,\infty);X_\bullet)}(u,v):=\sum_{n\in\N}\frac{1}{2^n}
 \frac{\sup\{\dd(u(t),v(t)):a\le t\le a+n\}}{1+\sup\{\dd(u(t),v(t)):a\le t
    \le a+n\}}.
\end{equation*}

Let 
\begin{equation*}
\TT:=
  \{I\subset\R:I=[T,\infty)\mathrm{~for~some~}T\in\R\}\cup\{(-\infty,\infty)\},
\end{equation*}
and for each $I\in\TT$, let $\FF(I)$ denote the set of all $X$-valued functions on $I$.

\begin{define}
\label{GES}
A map $\EE$ that associates to each $I\in\TT$ a subset $\EE(I)\subset\FF(I)$ will be called a generalized evolutionary system if the following conditions are satisfied:
\begin{enumerate}
\item $\EE([s,\infty))\ne\emptyset$ for each $s\in\R$.
\item $\{u(\cdot)|_I:u(\cdot)\in\EE(J)\}\subseteq\EE(I)$ for each $I,J\in\TT$ 
with $I\subseteq J$.
\item $\EE((-\infty,\infty))=\{u(\cdot):u(\cdot)|_{[s,\infty)}\in
\EE([s,\infty))~\forall T\in\R\}$.
\end{enumerate}
\end{define}
We refer to $\EE(I)$ as the set of all trajectories on the time interval $I$. Trajectories in $\EE((-\infty,\infty))$ are called complete. Next, for each $t\ge s\in\R$ and $A\subseteq X$, define the map 
\begin{equation*}
P(t,s):\PP(X)\rightarrow\PP(X),
\end{equation*}
\begin{equation*}
P(t,s)A:=\{u(t):u(s)\in A, u\in\EE([s,\infty))\}.
\end{equation*}
This map has the property that, for each $t\ge s\ge r\in\R$ and $A\subseteq X$,
\begin{equation*}
P(t,r)A\subseteq P(t,s)P(s,r)A.
\end{equation*}

We will also add the following assumption which is satisfied by the 3D Navier-Stokes equations (\cite{CK13}).
\begin{enumerate}
\refitem{A1}$\EE([s,\infty))$ is compact in $C([s,\infty);\Xw)$ for each $s\in\R$.
\end{enumerate}

\begin{define}\cite{CK13}
A family of sets $\AAd(t)\subseteq X$ ($t\in\R$) is a $\dd$-pullback attractor ($\bullet=$ s or w) if $\AAd(t)$ is a minimal set which is
\begin{enumerate}
\item $\dd$-closed.
\item $\dd$-pullback attracting: for any $B\subseteq X$ and any $\ee>0$, there exists $s_0:=s_0(B,\ee)<t\in\R$ so that for $s\le s_0$, 
\begin{equation*}
P(t,s)B\subseteq B_\bullet(\AAd(t),\ee):=\{u:\inf_{x\in\AAd(t)}\dd(u,x)<\ee\}.
\end{equation*}
\end{enumerate}
\end{define}

\begin{define}\cite{CK13}
The pullback omega limit set $\WWd$ ($\bullet =$ s or w) of a set $A\subseteq X$ is a family of sets given by 
\begin{equation*}
\WWd(A,t):=\bigcap_{s\le t}\overline{\bigcup_{r\le s}P(t,r)A}^\bullet.
\end{equation*}
\end{define}

Equivalently, $x\in \WWd(A,t)$ if there exist sequences $s_n\rightarrow-\infty$, $s_n\le t$, and $x_n\in P(t,s_n)A$, such that $x_n\xrightarrow{\dd}x$ as $n\rightarrow\infty$. 

Next, we recall the definition of the notion of invariance for a generalized evolutionary system. This requires the following mapping for $A\subseteq X$ and $s\le t\in\R$:
\begin{equation*}
\tilde{P}(t,s)A:=\{u(t):u(s)\in A, u\in\EE((-\infty,\infty))\}.
\end{equation*}

\begin{define}\cite{CK13}
A family of sets $\BB(t)\subseteq X$ is pullback semi-invariant if for each $s\le\in\R$,
\begin{equation*}
\tilde{P}(t,s)\BB(s)\subseteq \BB(t).
\end{equation*}
We say that $\BB(t)$ is pullback invariant if for each $s\le t\in\R$,
\begin{equation*}
\tilde{P}(t,s)\BB(s)=\BB(t).
\end{equation*}
$\BB(t)$ is pullback quasi-invariant if for each $b\in\BB(t)$, there exists a complete trajectory $u\in\EE((-\infty,\infty))$ with $u(t)=b$ and $u(s)\in \BB(s)$ for all $s\le t\in\R$.
\end{define}

\begin{thm}\cite{CK13}
\label{pullbackresults}
Let $\EE$ be a generalized evolutionary system. Then,
\begin{enumerate}
\item If the $\dd$-pullback attractor $\AAd(t)$ exists, then $\AAd(t)=\WWd(X,t)$.
\item The weak pullback attractor $\AAw(t)$ exists and is nonempty.
\end{enumerate}
Furthermore, if $\EE$ satisfies $\iref{A1}$, then
\begin{enumerate}
\addtocounter{enumi}{2}
\item $\AAw(t)=\WWw(X,t)=\WWs(X,t)=\{u(t):u\in\EE((-\infty,\infty))\}$.
\item $\AAw(t)$ is the maximal pullback invariant and maximal pullback quasi-invariant set.
\item (Weak pullback tracking property) Let $\ee > 0$ and $t \in \R$. There exists an $s_0 := s_0(\ee,t) \le t$ so that for all $s < s_0$ and any $u \in \EE([s,\infty))$ there exists a complete trajectory $v \in \EE((-\infty,\infty))$ with $\mathrm{d}_{C([s,\infty);\Xw)}(u,v) < \ee$.
\end{enumerate}
\end{thm}

\subsection{3D Navier-Stokes Equations}

Also in the paper \cite{CK13}, the authors applied the abstract framework of a generalized evolutionary system to the 3D Navier-Stokes equations with a translationally bounded forcing term. Here, we summarize the setup and major results. We will use this setup for the remainder of the paper. 

The 3D space-periodic, incompressible Navier-Stokes Equations (NSEs) on the periodic domain $\Omega:=\mathbb{T}^3$ are given by 
\begin{equation}
\label{NSE}
\left\{\begin{array}{l}\derivt u-\nu\Delta u+u\cdot\grad u+\grad p=f(t) \\
                       \grad\cdot u=0\end{array}\right.
\end{equation}
where $u$, the velocity vector field, and $p$, the pressure, are unknowns; $\nu>0$ is the kinematic viscosity of the fluid, and $f(t)\in L^2_{loc}(\R,H^{-1}(\Omega)^3)$ is a time-dependent forcing term. Assume that the initial condition $u(\cdot,s)$ and the forcing term $f(t)$ have the property that 
\begin{equation*}
\displaystyle\int_\Omega u(x,s)\ddx=\int_\Omega f(x,t)\ddx =0,
\end{equation*}
for all $t$. Then, we have that 
\begin{equation*}
\int_\Omega u(x,t)\ddx=0
\end{equation*}
for all $t\ge s\in\R$. The functional setting is given below.

Denote by $\ip{\cdot}{\cdot}$ and $\abs{\cdot}$ the $L^2(\Omega)^3$-inner product and the $L^2(\Omega)^3$-norm, respectively. Let $\mathscr{V}$ be given by
\begin{equation*}
\mathscr{V}:=\left\{u\in [C^\infty(\Omega)]^3:\int_\Omega u(x)\ddx
                       =0,~\nabla\cdot u=0\right\}.
\end{equation*}
Next, let $H$ and $V$ be the closures of $\mathscr{V}$ in $L^2(\Omega)^3$ and $H^1(\Omega)^3$, respectively. Denote by $H_w$ the set $H$ endowed with the weak topology.

Let $P_\sigma:L^2(\Omega)^3\rightarrow H$ be the $L^2$ orthogonal projection, known as the Leray projector. Let $A:=-P_\sigma\Delta=-\Delta$ be the Stokes operator with domain $(H^2(\Omega))^3\cap V$. Note that the Stokes operator is a self-adjoint, positive operator with compact inverse. Let 
\begin{equation*}
\norm{u}:=\abs{A^{1/2}u}.
\end{equation*}
Note that $\norm{u}$ is equivalent to the $H^1$ norm of $u$ for $u\in D(A^{1/2})$ by the Poincar\'{e} inequality. That is, $\abs{A^{(k-1)/2}u} \le \lambda_1^{-1/2} \abs{A^{k/2}u}$ where $\lambda_1$ is the first eigenvalue of the Stokes operator $A$ and $k \in \R$. Let $\IP{\cdot}{\cdot}$ denote the corresponding inner product in $H^1$.

Next, denote by $B(u,v):=P_\sigma(u\cdot\nabla v)\in V'$ for each $u,v\in V$. This is a bilinear form with the following property:
\begin{equation*}
\angip{B(u,v)}{w}=-\angip{B(u,w)}{v}
\end{equation*}
for each $u,v,w\in V$.

We can now rewrite (\ref{NSE}) as a differential equation in $V'$. That is, 
\begin{equation}
\label{funcNSE}
\derivt u+\nu Au+B(u,u)=g
\end{equation}
for $g:=P_\sigma f$, and $u$ is a $V$-valued function of time.

\begin{define}
\label{weaksoln}
The function $u:[s,\infty)\rightarrow H$ (or $u:(-\infty,\infty)\rightarrow H$) is a weak solution to (\ref{NSE}) on $[s,\infty)$ (or $(-\infty,\infty)$) if
\begin{enumerate}
\item $\derivt u\in L_{loc}^1([s,\infty);V')$.
\item $u\in C([s,\infty);H_w)\cap L^2_{loc}([s,\infty);V)$.
\item $\ip{\derivt u(t)}{\phi}+\nu\IP{u(t)}{\phi}+\angip{B(u(t),u(t))}{\phi}
           =\angip{g(t)}{\phi}$ for a.e. $t\in[s,\infty)$ and each $\phi\in V$.
\end{enumerate}
\end{define}

\begin{thm}[Leray, Hopf]
\label{LHsolns}
For each $u_0\in H$ and $g\in L^2_{loc}(\R;V')$, there exists a weak solution of (\ref{NSE}) on $[s,\infty)$ with $u(s)=u_0$, and for each $t\ge t_0$, $t_0$ a.e. in $[s,\infty)$ we have the following energy inequality:
\begin{equation}
\label{LHenergyineq}
\abs{u(t)}^2+2\nu\int_{t_0}^t\norm{u(\xi)}^2\ddxi\le\abs{u(t_0)}^2+2\int_{t_0}^t
   \angip{g(\xi)}{u(\xi)}\ddxi.
\end{equation}
\end{thm}

\begin{define}
\label{LHweaksoln}
A weak solution to (\ref{NSE}) satisfying (\ref{LHenergyineq}) will be called a Leray-Hopf weak solution.
\end{define}

\begin{define}
\label{leraysoln}
A Leray-Hopf weak solution to (\ref{NSE}) on $[s,\infty)$ satisfying the energy inequality 
\begin{equation*}
\abs{u(t)}^2+2\nu\int_s^t\norm{u(\xi)}^2\ddxi\le\abs{u(s)}^2+2\int_s^t
                                               \angip{g(\xi)}{u(\xi)}\ddxi
\end{equation*}
for each $t\ge s$ is called a Leray solution.
\end{define}

A Leray solution is a specific type of Leray-Hopf solution which is strongly continuous at the starting time $s$. A classical argument shows that for any $u_0 \in H$ and any $s \in \R$, one can build a Leray solution starting at time $s$ with initial data $u_0$ using a Galerkin technique.

Fix $\tau>0$. Assume $g$ is translationally bounded in $L^2_{\mathrm{loc}}(\R,V')$. That is, 
\begin{equation*}
\norm{g}_{L^2_b(\tau)}^2:=\sup_{t\in\R}\frac{1}{\tau}\int_t^{t+\tau}
        \norm{g(\xi)}^2_{V'}\ddxi<\infty.
\end{equation*}
First, note that $\norm{g}_{V'}$ and $\norm{g}_{L^2_b(\tau)}$ have the same dimensions. Next, note that the choice of $\tau$ is not particularly important. In fact, for any $\tau,\rho>0$, we have that the norms $\norm{\cdot}_{L^2_b(\tau)}$ and $\norm{\cdot}_{L^2_b(\rho)}$ are equivalent. 

\begin{lem}
Let $\tau,\rho>0$ be given. Assume, without loss of generality that $\tau\le\rho$. Then, for any translationally bounded $g\in L^2_{loc}(\R,V')$,  
\begin{equation*}
\frac{\tau}{\rho}\norm{g}_{L^2_b(\tau)}^2\le
  \norm{g}_{L^2_b(\rho)}^2\le
  \frac{N\tau}{\rho}\norm{g}_{L^2_b(\tau)}^2,
\end{equation*}
where $N$ is any integer so that $N\tau\ge\rho$.
\end{lem}

The proof is elementary and is thus omitted. Therefore, we may use whatever $\tau>0$ we like in our calculations. Later, we will choose $\tau:=(\nu\lambda_1)^{-1}$. 

As was shown in \cite{CK13}, there exists an absorbing ball for Leray solutions of (\ref{funcNSE}). That is, from the energy inequality and the fact that $g$ is translationally bounded, one can derive the following inequality: 
\begin{equation*}
\abs{u(t)}^2 \le \abs{u(t_0)}^2 e^{\nu\lambda_1(t_0-t)}
    +\frac{\tau\norm{g}_{L^2_b(\tau)}^2}{\nu(1-e^{-\nu\lambda_1\tau})}
\end{equation*}
for almost every $t_0\geq s$ (including $t_0 = s$) and all $t>t_0$. 

Letting 
\begin{equation}
\label{absrad}
R:=\frac{2\tau\norm{g}^2_{L^2_b(\tau)}}{\nu(1-e^{-\nu\lambda_1\tau})}, 
\end{equation}
we define 
\begin{equation*}
X:=\{u \in H:\abs{u}^2\le R\}
\end{equation*}
as a closed absorbing ball in $H$. In particular, $X$ is weakly compact with strong and weak metrics given by 
\begin{equation*}
\ds(u,v):=\abs{u-v} \quad \text{ and } \quad \dw(u,v):=\sum_{k\in\Z^3}
   \frac{1}{2^{\abs{k}}}
       \frac{\abs{\hat{u_k}-\hat{v_k}}}{1+\abs{\hat{u_k}-\hat{v_k}}}
\end{equation*}
for $u,v\in H$ where $\hat{u_k}$ and $\hat{v_k}$ are the Fourier coefficients of $u$ and $v$, respectively. Note that the above weak metric $\dw$ induces the weak topology on $X$. 

Next, we define our generalized evolution system on $X$ by 
\begin{align*}
\EE([s,\infty)):=&\{u:u\mathrm{~is~a~Leray{-}Hopf~solution~of~(\ref{funcNSE})~
            on~}[s,\infty) \\
                 &\mathrm{~and~}u(t)\in X\mathrm{~for~}t\in[s,\infty)\}, \\
\EE((-\infty,\infty)):=&\{u:u\mathrm{~is~a~Leray{-}Hopf~solution~of~
       (\ref{funcNSE})~on~}(-\infty,\infty) \\
                 &\mathrm{~and~}u(t)\in X\mathrm{~for~}t\in(-\infty,\infty)\}.
\end{align*}
Then, $\EE$ satisfies the necessary properties in Definition~\ref{GES} and forms a generalized evolutionary system on $X$.

As observed in \cite{CK13}, we must use Leray-Hopf weak solutions in the definition of our evolutionary system since the restriction of a Leray solution is not necessarily a Leray solution. However, the restriction of a Leray solution is always a Leray-Hopf weak solution. In fact, $\EE$ satisfies \iref{A1}. Therefore, by Theorem~\ref{pullbackresults}, we have the following theorem.

\begin{thm}\cite{CK13}
\label{pullbacknse}
Let $g$ be translationally bounded in $L^2_{loc}(\R,V')$. Then, there exists a weak pullback attractor $\AAw(t)$ for the generalized evolutionary system $\EE$ of Leray-Hopf weak solutions to (\ref{funcNSE}). In particular $\EE$ satisfies $\iref{A1}$. Therefore,  
\begin{equation*}
\AAw(t)=\{u(t):u\in\EE((-\infty,\infty))\}
\end{equation*}
is the maximal invariant and quasi-invariant subset of $X$.
\end{thm}

In particular, there exists a complete bounded (in the sense of $H$) weak solution to the 3D Navier-Stokes equations. In the next section, we will present an argument demonstrating that when the force is small enough, the weak pullback attractor consists of only one such solution. In this case, we have that 
\begin{equation*}
\AAw(t)=\{u(t)\}, 
\end{equation*}
is trivial. 

\section{Degenerate Pullback Attractors}
\label{trivial}
\subsection{A Criterion for Strong Solutions}
\label{strong}

In our goal of proving that the pullback attractor consists of a single point, we will begin by showing that if the force is small enough, then a complete bounded solution guaranteed by Theorem~\ref{pullbacknse} is, in fact, a strong solution.

\begin{define}
\label{strongsoln}
A weak solution $u$ to (\ref{funcNSE}) will be called strong if $u \in L^\infty_{loc}(\R,V)$.
\end{define}

Let $v$ be a complete bounded solution to (\ref{funcNSE}) as discussed in the previous section. In particular, $v$ satisfies the inequality (\ref{vineq}). Using the Cauchy-Schwarz inequality followed by Young's inequality, we find that 
\begin{equation}
\abs{v(t)}^2 + \nu\int_s^t\norm{v(\xi)}^2\ddxi \le \abs{v_0}^2 + 
   \frac{1}{\nu}\int_s^t\norm{g(\xi)}_{V'}^2\ddxi.
\end{equation}
Using the radius of the absorbing ball given in (\ref{absrad}) and dropping the first term on the left-hand side, we find that 
\begin{equation*}
\nu\int_s^t\norm{v(\xi)}^2\ddxi\le\frac{2\tau\norm{g}_{L^2_b(\tau)}^2}
                                       {\nu(1-e^{-\nu\lambda_1\tau})}
   + \frac{1}{\nu}\int_s^t\norm{g(\xi)}_{V'}^2\ddxi.
\end{equation*}
Thus, we find that for any $s\in\R$ 
\begin{equation}
\label{gradineq1}
\int_s^{s+\tau}\norm{v(\xi)}^2\ddxi
 \le \frac{\tau\norm{g}_{L^2_b(\tau)}^2(3-e^{-\nu\lambda_1\tau})}
          {\nu^2(1-e^{-\nu\lambda_1\tau})}.
\end{equation}

Hence, for any $M \ge 0$, 
\begin{equation*}
\left|\{x\in[s,s+\tau]:\norm{v(x)} \ge M\}\right| \le \frac{1}{M^2}
  \frac{\tau\norm{g}_{L^2_b(\tau)}^2(3-e^{-\nu\lambda_1\tau})}
       {\nu^2(1-e^{-\nu\lambda_1\tau})}.
\end{equation*}
Letting $M := \left(\frac{2\norm{g}_{L^2_b(\tau)}^2(3-e^{-\nu\lambda_1\tau})}{\nu^2(1-e^{-\nu\lambda_1\tau})}\right)^{1/2}$, we have that 
\begin{equation*}
\left|\{x\in[s,s+\tau]:\norm{v(x)} \ge M\}\right| \le \frac{\tau}{2}.
\end{equation*}
We encapsulate the above remarks into the following lemma. 

\begin{lem}
\label{pointbound}
Let $v$ be any complete, bounded solution to (\ref{funcNSE}) with $g$ translationally bounded in $L^2_{loc}(\R,V')$ whose existence is guaranteed by Theorem~\ref{pullbacknse}. Then, for any $s \in \R$, there exists a point $t \in [s,s+\tau]$ so that  
\begin{equation*}
\norm{v(t)}^2
  \le \frac{2\norm{g}_{L^2_b(\tau)}^2(3-e^{-\nu\lambda_1\tau})}
           {\nu^2(1-e^{-\nu\lambda_1\tau})}<\infty.
\end{equation*}
\end{lem}

Now, we add the assumption that $g$ is translationally bounded in \\ $L^2_{loc}(\R,H)$ which will be assumed for the remainder of the paper. That is, we assume that  
\begin{equation*}
\norm{g}_{L^2_0(\tau)}^2:=\sup_{t\in\R}\frac{1}{\tau}\int_t^{t+\tau}
       \abs{g(\xi)}^2\ddxi<\infty.
\end{equation*}
Note that using the Poincar\'{e} inequality, we have that 
\begin{equation*}
\norm{g}_{L^2_b(\tau)}^2\le\lambda_1^{-1}\norm{g}_{L^2_0(\tau)}^2.
\end{equation*}
We will show that if $\norm{g}_{L^2_0(\tau)}$ is sufficiently small, then $v \in L^\infty(\R,V)$. 

To do this, let $t_0\in\R$ be arbitrary. Then, consider the interval $[t_0-\tau,t_0]$. By Lemma~\ref{pointbound}, there exists a point $t \in [t_0-\tau,t_0]$ so that 
\begin{equation*}
\norm{v(t)}^2
  \le\frac{2\norm{g}_{L^2_b(\tau)}^2(3-e^{-\nu\lambda_1\tau})}
          {\nu^2(1-e^{-\nu\lambda_1\tau})}<\infty.
\end{equation*}
Thus, by Leray's characterization \cite{L34}, there is an $\ee > 0$ so that $v$ is a strong solution on $[t,t+\ee)$. We investigate the length of this interval. 

Starting with (\ref{funcNSE}), we take the inner product with $Av$ giving us that 
\begin{equation}
\label{grad1}
\frac{1}{2}\derivt\norm{v}^2+\nu\abs{Av}^2 
   \le \abs{\ip{B(v,v)}{Av}}+\abs{\ip{g}{Av}}.
\end{equation}
Classical esimates give us that 
\begin{align}
\label{sobolevconst}
\abs{\ip{B(v,v)}{Av}} &\le c_0 \norm{v}^{3/2}\abs{Av}^{3/2} \\
\abs{\ip{g}{Av}}      &\le \abs{g}\abs{Av}.
\end{align}
Next, we apply Young's inequality on each of these terms to get that 
\begin{align*}
\abs{\ip{B(v,v)}{Av}} &\le \frac{\nu}{4}\abs{Av}^2 + \frac{c_0}{\nu^3}\norm{v}^6 \\
\abs{\ip{g}{Av}}      &\le \frac{1}{\nu}\abs{g}^2 + \frac{\nu}{4}\abs{Av}^2.
\end{align*}
Using these estimates as well as the Poincar\'{e} inequality, (\ref{grad1}) reduces to 
\begin{equation}
\label{grad2}
\derivt\norm{v}^2 + \nu\lambda_1\norm{v}^2
  \le \frac{2}{\nu}\abs{g}^2 + \frac{c_0}{\nu^3}\norm{v}^6.
\end{equation}

Now, assume that 
\begin{equation*}
\norm{g}_{L^2_0(\tau)}^2
  \le \frac{c_0^{-1/2}\nu^4\lambda_1^{3/2}}{2c_1+4\nu\lambda_1\tau}
\end{equation*}
where $c_1:=\frac{2(3-e^{-\nu\lambda_1\tau})}{1-e^{-\nu\lambda_1\tau}}$. Then, we will show that $\norm{v(t_0)}^2\le c_0^{-1/2}\nu^2\lambda_1^{1/2}$. The following is a modification of the argument given in \cite{cf88}. For completeness, we present the argument in its entirety. 

First, note that the criterion on $\norm{g}_{L^2_0(\tau)}$ guarantees that 
\begin{align*}
\norm{v(t)}^2+\frac{2}{\nu}\int_t^{t+\tau}\abs{g(\xi)}^2\ddxi
  &\le \frac{c_1}{\nu^2}\norm{g}_{L^2_b(\tau)}^2
         +\frac{2\tau}{\nu}\norm{g}_{L^2_0(\tau)}^2 \\
  &\le \frac{c_1}{\nu^2\lambda_1}\norm{g}_{L^2_0(\tau)}^2
        +\frac{2\tau}{\nu}\norm{g}_{L^2_0(\tau)}^2 \\
  &\le \frac{c_0^{-1/2}\nu^2\lambda_1^{1/2}}{2}.
\end{align*}
Then, certainly $\norm{v(t)}^2< c_0^{-1/2}\nu^2\lambda_1^{1/2}$. Let 
\begin{equation*}
T:=\sup\{T_0 \in [t,t+\tau]:\norm{v(T_0)}^2 < c_0^{-1/2}\nu^2\lambda_1^{1/2}\}.
\end{equation*}
Since $v$ is a strong solution at $t$ we get that $T>t$. Assume that $T<t+\tau$. Using $\norm{v(T_0)}^2 < c_0^{-1/2}\nu^2\lambda_1^{1/2}$ for each $T_0\le T$, we find that 
\begin{equation*}
\nu\lambda_1\norm{v(T_0)}^2 - \frac{c_0}{\nu^3}\norm{v(T_0)}^6 
  = \nu\lambda_1\norm{v(T_0)}^2\left(1-\frac{c_0}{\nu^4\lambda_1}
               \norm{v(T_0)}^4\right) \ge 0.
\end{equation*}
Thus, we integrate (\ref{grad2}) from $t$ to $T$ and get that 
\begin{align*}
\norm{v(T)}^2
 &\le \norm{v(t)}^2 + \frac{2}{\nu}\int_t^T\abs{g(\xi)}^2\ddxi \\
 &\le \norm{v(t)}^2 + \frac{2}{\nu}\int_t^{t+\tau}\abs{g(\xi)}^2\ddxi \\
 &\le \frac{c_0^{-1/2}\nu^2\lambda_1^{1/2}}{2}.
\end{align*}
Thus, we must have that $T=t+\tau$. In particular, this is true of $t_0\in[t,t+\tau]$. Since $t_0\in\R$ was arbitrary, we have that 
\begin{equation}
\label{grad3}
\norm{v(t)}^2 < c_0^{-1/2}\nu^2\lambda_1^{1/2}
\end{equation}
for all $t \in \R$. This completes the proof of the following theorem.

\begin{thm}
\label{regularity}
Suppose $g$ is translationally bounded in $L^2_{loc}(\R,H)$ so that 
\begin{equation*}
\norm{g}_{L^2_0(\tau)}^2\le \frac{c_0^{-1/2}\nu^4\lambda_1^{3/2}}
                                 {2c_1+4\nu\lambda_1\tau}
\end{equation*}
for $c_1:=\frac{2(3-e^{-\nu\lambda_1\tau})}{1-e^{-\nu\lambda_1\tau}}$ and $c_0$ the constant given in (\ref{sobolevconst}). Then, there exists a complete, bounded, strong solution to (\ref{funcNSE}) so that $v \in L^\infty(\R,V)$. In particular, $\norm{v(t)}^2 < c_0^{-1/2}\nu^2\lambda_1^{1/2}$ for all $t \in \R$.
\end{thm}

We now let $\tau:=(\nu\lambda_1)^{-1}$. For simplicity, we set 
\begin{align*}
\norm{g}_{L^2_b((\nu\lambda_1)^{-1})}&=:\norm{g}_{L^2_b} \\
\norm{g}_{L^2_0((\nu\lambda_1)^{-1})}&=:\norm{g}_{L^2_0}.
\end{align*}
Then, we can express Theorem~\ref{regularity} in terms of the non-dimensional 3D Grashof number
\begin{equation*}
G:= \frac{\norm{g}_{L^2_0}}{\nu^2\lambda_1^{3/4}}.
\end{equation*}

\begin{cor}
\label{regularity2}
Suppose $g$ is translationally bounded in $L^2_{loc}(\R,H)$ so that 
\begin{equation*}
G^2 = \frac{\norm{g}_{L^2_0}^2}{\nu^4\lambda_1^{3/2}}\le\frac{c_0^{-1/2}}{2c_1+4}
\end{equation*}
for $c_1:=\frac{2(3-e^{-1})}{1-e^{-1}}$ and $c_0$ the constant given in (\ref{sobolevconst}). Then, there exists a complete, bounded, strong solution to (\ref{funcNSE}) so that $v \in L^\infty(\R,V)$. In particular, $\norm{v(t)}^2 < c_0^{-1/2}\nu^2\lambda_1^{1/2}$ for all $t \in \R$.
\end{cor}

It is also worthwile to note that the above argument proves the strongness of all complete trajectories in our generalized evolutionary system $\EE$. In fact, it proves that if $u \in \EE([s,\infty))$, then for $t > s+\tau$, $u:[t,\infty) \rightarrow V$ is a strong solution.

\subsection{A Serrin-type Argument}
\label{serrin}

In (\cite{S62}), Serrin presents an argument for the uniqueness of weak solutions in an interval of regularity (where a strong solution exists). Using a modification of the argument as it is presented in (\cite{T84}), we obtain the required argument for the existence of degenerate pullback attractors.

Let $v$ be a complete, bounded strong solution to (\ref{funcNSE}) on $(-\infty,\infty)$ guaranteed by Theorem~\ref{regularity}. Let $u$ be another Leray-Hopf weak solution to (\ref{funcNSE}) on $[T,\infty)$ and let $w:=u-v$. Then, $u$ and $v$ satisfy 
\begin{align}
\abs{u(t)}^2 + 2\nu\int_s^t\norm{u(\xi)}^2\ddxi 
  &\le \abs{u_0}^2 + 2\int_s^t\angip{g(\xi)}{u(\xi)}\ddxi \label{uineq}, \\ 
\abs{v(t)}^2 + 2\nu\int_s^t\norm{v(\xi)}^2\ddxi 
  &= \abs{v_0}^2+2\int_s^t\angip{g(\xi)}{v(\xi)}\ddxi \label{vineq},
\end{align}
respectively for a.a. $s\geq T$, all $t\geq s$,  with $u_0 := u(s)$ and $v_0 := v(s)$. Also, as seen in Temam's book \cite{T84}
\begin{align*}
\ip{u(t)}{v(t)} + 2\nu\int_s^t\IP{u(\xi)}{v(\xi)}\ddxi 
                       \numberthis \label{weakdifference}
           =& \ip{u(s)}{v(s)} \\
           &+   \int_s^t\angip{g(\xi)}{u(\xi)+v(\xi)}\ddxi \\
           &-   \int_s^t\angip{B(w(\xi),w(\xi))}{v(\xi)}\ddxi.
\end{align*}

Adding (\ref{uineq}) to (\ref{vineq}) and then subtracting twice (\ref{weakdifference}), we get that
\begin{equation}
\label{wineq}
\abs{w(t)}^2 + 2\nu\int_s^t\norm{w(\xi)}^2\ddxi
   \le \abs{w(s)}^2 + 2\int_s^t\angip{B(w(\xi),w(\xi))}{v(\xi)}\ddxi.
\end{equation}

We estimate the nonlinear term using classical estimates. That is, we find that 
\begin{align*}
\abs{\angip{B(w,w)}{v}}
  \le C\abs{w}^{1/4}\norm{w}^{7/4}
         \abs{v}^{1/4}\norm{v}^{3/4} \\
  \le \frac{\nu}{2}\norm{w}^2 + \frac{C}{\nu^7}
         \abs{v}^2\norm{v}^6\abs{w}^2
\end{align*}
after applying the Young's inequality. Since $v\in L^{\infty}(\R,H)\cap L^{\infty}(\R,V)$, we use (\ref{absrad}) and (\ref{grad3}) to estimate (\ref{wineq}) by 
\begin{equation}
\label{grashof}
\abs{w(t)}^2-\abs{w(s)}^2
  \le \nu\lambda_1\int_s^t\left(C\frac{\tau\norm{g}_{L^2_0(\tau)}^2}
                                     {\nu^3\lambda_1^{1/2}}
                               -1\right)\abs{w(\xi)}^2\ddxi.
\end{equation}

Assuming that $\norm{g}^2_{L^2_0(\tau)}$ is sufficiently small, we can ensure that \\ $C\tau\norm{g}^2_{L^2_0(\tau)}<\nu^3\lambda_1^{1/2}$ giving us that 
\begin{equation}
\label{wineq2}
\abs{w(t)}^2 - \abs{w(s)}^2 \le -M\int_s^t\abs{w(\xi)}^2\ddxi
\end{equation}
for $M := \nu\lambda_1\left(1-C\frac{\tau\norm{g}^2_{L^2_0(\tau)}}
{\nu^3\lambda_1^{1/2}}\right)>0$. 
Thus, after applying Gronwall's inequality, we have that 
\begin{equation*}
\abs{w(t)}^2 \le \abs{w(s)}^2e^{M(s-t)}.
\end{equation*}
In particular, if $u$ is also a complete bounded Leray-Hopf solution, i.e., $T=-\infty$ and $u\in \EE((-\infty,\infty))$, then for each $t$ fixed we can take a limit as $s\rightarrow-\infty$
obtaining  $\abs{w(t)} =0$, i.e, $u(t)=v(t)$. This completes the proof of the following theorem.

\begin{thm}
\label{singleton}
Let $g$ be translationally bounded in $L^2_{loc}(\R,H)$. Assume that $\norm{g}_{L^2_0(\tau)}$ is sufficiently small, then the weak pullback attractor for (\ref{funcNSE}) is a single point, 
\begin{equation*}
\AAw(t)=\{v(t)\}
\end{equation*}
for some complete, bounded, strong solution to (\ref{funcNSE}).
\end{thm}

Again, if we let $\tau := (\nu\lambda_1)^{-1}$, then (\ref{grashof}) simplifies to 
\begin{equation}
\label{grashof2}
\abs{w(t)}^2 - \abs{w(s)}^2
  \le \nu\lambda_1\int_s^t(CG^2-1)\abs{w(\xi)}^2\ddxi.
\end{equation}
So, we can restate (\ref{singleton}) once again in terms of the 3D Grashof constant.

\begin{cor}
Let $g$ be translationally bounded in $L^2_{loc}(\R,H)$. Assume that the Grashof number $G$ given by 
\begin{equation*}
G=\frac{\norm{g}_{L^2_0}}{\nu^2\lambda_1^{3/4}}
\end{equation*}
is sufficiently small, then the weak pullback attractor for (\ref{funcNSE}) is a single point, 
\begin{equation*}
\AAw(t)=\{v(t)\}
\end{equation*}
for some complete, bounded, strong solution to (\ref{funcNSE}).
\end{cor}

\begin{rmk}
Theorem~\ref{singleton} together with Theorem~\ref{pullbacknse} imply that there exists a unique complete bounded Leray-Hopf solution of the 3D Navier-Stokes equations provided the force is small enough.
\end{rmk}
\subsection{Periodic Force}
\label{periodic}

The existence of a unique periodic solution to the 3D Navier-Stokes equations is a remarkable consequence of this Theorem. To begin, let the force $f$ in (\ref{NSE}) be periodic in $L^2_{loc}(\R,L^2(\Omega)^3)$ with period $\rho$. Then, the projected force $g$ in (\ref{funcNSE}) is also periodic in $L^2_{loc}(\R,H)$ with period $\rho$. A straightforward argument shows that $g$ is translationally bounded. Thus, by Theorem~\ref{singleton}, if $g$ is sufficiently small, there exists a unique complete bounded Leray-Hopf solution $v(t)$ to (\ref{funcNSE}), which
is actually a strong solution. We will show that $v(t)$ is, in fact, periodic.

\begin{thm}
\label{periodicforce}
Let $g$ be periodic in $L^2_{loc}(\R,H)$ with period $\rho$. Assume that $g$ is sufficiently small. Then, there exists a unique, periodic, strong solution $v(t)$ to (\ref{funcNSE}). In particular, $v(t)$ has period $\rho$.
\end{thm}

\begin{proof}
Due to Theorem~\ref{singleton}, we only must show that the unique complete bounded solution $v(t)$ has period $\rho$. To this end, note that $v(t)$ satisfies the equation 
\begin{equation}
\label{periodic1}
\derivt v(t)+\nu Av(t)+B(v(t),v(t))=g(t).
\end{equation}
Then, of course, $v$ satisfies 
\begin{equation*}
\derivt v(t+\rho)-\nu Av(t+\rho)+B(v(t+\rho),v(t+\rho))=g(t+\rho).
\end{equation*}
But, $g(t+\rho)=g(t)$. So, $v(\cdot+\rho)$ also satisfies (\ref{periodic1}). By uniqueness of a complete bounded solution, $v(t+\rho)=v(t)$.
\end{proof}

\bibliographystyle{abbrv}
\bibliography{refs}

\end{document}